\documentclass[12pt]{article}
\usepackage[english]{babel}
\usepackage{amsmath,amsthm}
\usepackage{amsfonts}
\usepackage{enumerate}
\usepackage{hyperref}
\hypersetup{colorlinks,
             linkcolor=blue, anchorcolor=blue,
             citecolor=blue, urlcolor=blue}
\newtheorem{thm}[equation]{Theorem}

\newtheorem{lem}[equation]{Lemma}
\newtheorem{prop}[equation]{Proposition}

\theoremstyle{definition}
\newtheorem{defn}[equation]{Definition}
\newtheorem{remark}[equation]{Remark}
\newtheorem{observation}[equation]{Observation}

\newtheorem*{notation}{Notation}

\numberwithin{equation}{section}
\begin{document}

\title{\bf\Large A uniqueness theorem for meromorphic maps into $\mathbb{P}^n$ with generic $(2n+2)$ hyperplanes}%
\author{Kai Zhou \footnote{{\it E-mail address}: \texttt{zhoukai@tongji.edu.cn}}}%
\date{}%
\maketitle
\def\thefootnote{}
\begin{abstract}
  Let $ H_1,\dots,H_{2n+2}$ be \emph{generic} $(2n+2)$ hyperplanes in $\mathbb{P}^n.$ It is proved that if meromorphic maps $ f $ and $ g $ of $\mathbb{C}^m $ into $\mathbb{P}^n $ satisfy $ f^*(H_j)=g^*(H_j)$ $(1\leq j\leq 2n+2)$ and $ g $ is algebraically non-degenerate then $ f=g.$ This result is essentially implied by the proof of Hirotaka Fujimoto in papers [Nagoya Math. J., 1976(64): 117--147] and [Nagoya Math. J., 1978(71): 13--24]. This note gives a complete proof of the above uniqueness result. 
\end{abstract}
\section{Introduction}             \label{sec:Introduction}

In the 1970s, Hirotaka Fujimoto \cite{Fujimoto75, Fujimoto76, Fujimoto78} considered the generalization of the uniqueness theorems of meromorphic functions obtained in \cite{Polya1921} and \cite{Nevanlinna1926} to the case of meromorphic maps into a complex projective space. Fujimoto considered two meromorphic maps $ f $ and $ g $ of $\mathbb{C}^m $ into the $ n$-dimensional complex projective space $\mathbb{P}^n $ that satisfy $ f^*(H_j)=g^*(H_j)$ $(1\leq j\leq q)$ for $ q $ hyperplanes $ H_1,\dots, H_q $ in $\mathbb{P}^n $ in general position, where $ f^*(H_j)$ and $ g^*(H_j)$ are the pullback divisors of $ H_j $ by $ f $ and by $ g $ respectively. Under some additional conditions, it can be proved that $ f=g.$ For instance, Fujimoto proved in \cite{Fujimoto76, Fujimoto78} that if $ g $ is algebraically non-degenerate and $ q\geq 2n+3 $ then $ f=g.$

I found that Fujimoto's proof in the papers \cite{Fujimoto76,Fujimoto78} essentially implies the following uniqueness result.
\begin{thm}           \label{thm:generic_(2n+2)}
Let $ H_1,\dots,H_{2n+2}$ be \emph{generic} $(2n+2)$ hyperplanes in $\mathbb{P}^n.$ Let $ f $ and $ g $ be meromorphic maps of $\mathbb{C}^m $ into $\mathbb{P}^n $ such that $ f(\mathbb{C}^m)\not\subseteq H_j $, $ g(\mathbb{C}^m)\not\subseteq H_j $ and $ f^*(H_j)=g^*(H_j)$ $(1\leq j\leq 2n+2).$ If $ g $ is algebraically non-degenerate, then $ f=g.$
\end{thm}

I shall give an explanation of the word ``generic''. Denote by $(\mathbb{P}^n)^*$ the set of all hyperplanes in $\mathbb{P}^n.$ We identify $(\mathbb{P}^n)^*$ with $\mathbb{P}^n $ in a natural way. The ``generic'' in the condition of Theorem \ref{thm:generic_(2n+2)} means that the point $(H_1,\dots,H_{2n+2})\in \underset{(2n+2)\,\mbox{copies}}{\underbrace{(\mathbb{P}^n)^*\times\cdots\times(\mathbb{P}^n)^*}}$ is not in $ V,$ where $ V $ is a pre-given proper algebraic subset of $\underset{(2n+2)\,\mbox{copies}}{\underbrace{(\mathbb{P}^n)^*\times\cdots\times(\mathbb{P}^n)^*}}.$

Theorem \ref{thm:generic_(2n+2)} is a new and interesting result, so I write this note to state it and also give a complete proof of it. The proof of Theorem \ref{thm:generic_(2n+2)} is given in Section \ref{sec:uniqueness_theorem_generic_(2n+2)} and most parts of it use the same idea as in Fujimoto's proof, but the presentation is slightly different from Fujimoto's.

I emphasize that this manuscript is just a note but not an original research paper.

This note is organized as follows. In Section \ref{sec:Preliminaries_auxiliary results}, we recall some basic definitions and include some auxiliary results. In particular, we shall include some combinatorial lemmas due to Fujimoto. In Section \ref{sec:Vfg}, we prove a theorem concerning the dimension of the Zariski closure of the image of $ f\times g $ in $\mathbb{P}^n\times\mathbb{P}^n.$ This theorem and its proof are used essentially in Section \ref{sec:uniqueness_theorem_generic_(2n+2)}. In Section \ref{sec:Two_propositions}, we prove two propositions which will be used in Section \ref{sec:uniqueness_theorem_generic_(2n+2)}. In Section \ref{sec:uniqueness_theorem_generic_(2n+2)}, we prove the uniqueness theorem with generic $(2n+2)$ hyperplanes.

\section{Preliminaries and auxiliary results}          \label{sec:Preliminaries_auxiliary results}

\paragraph{2.1\, Meromorphic maps and hyperplanes in complex projective spaces.}

Let $ f $ be a meromorphic map of $\mathbb{C}^m $ into $\mathbb{P}^n.$ Denote by $ I(f)$ the indeterminacy locus of $ f,$ which is an analytic subset of $\mathbb{C}^m $ of codimension $\geq 2.$ A \emph{reduced representation} of $ f $ is a $(n+1)$-tuple $(f_0,\dots, f_n)$ of holomorphic functions on $\mathbb{C}^m $ such that
\[
   \{f_0=f_1=\dots=f_n=0\}= I(f)
\]
and
\[
   f(z)=[f_0(z):\dots:f_n(z)] \quad \forall z\in\mathbb{C}^m\setminus I(f).
\]
Any meromorphic map of $\mathbb{C}^m $ into $\mathbb{P}^n $ has a reduced representation.

We also recall the following definition.
\begin{defn}        \label{defn:linear_and_algebraic_nondegeneracy}
A meromorphic map of $\mathbb{C}^m $ into $\mathbb{P}^n $ is said to be \emph{linearly} (resp. \emph{algebraically}) \emph{non-degenerate}, if its image is not contained in any hyperplane (resp. hypersurface) in $\mathbb{P}^n.$
\end{defn}

Let
\[
   H=\big\{[x_0:\dots:x_n]\in\mathbb{P}^n\,|\, a_0x_0+\dots+a_nx_n=0\big\}
\]
be a hyperplane in $\mathbb{P}^n $ where $(a_0,\dots,a_n)\in\mathbb{C}^{n+1}\setminus\{(0,\dots,0)\}.$ We shall say that the linear form $ a_0X_0+\dots+a_nX_n $ defines the hyperplane $ H.$

\begin{defn}    \label{defn:hyperplanes_in_general_position}
Let $\{H_j\}_{j=1}^q $ be a family of hyperplanes in $\mathbb{P}^n.$ Take a linear form $ L_j=a^j_0 X_0+\dots+a^j_n X_n $ that defines $ H_j $ for each $ 1\leq j\leq q.$ The family $\{H_j\}_{j=1}^q $ of hyperplanes is said to be \emph{in general position}, if, for any distinct indices $ i_1,\dots, i_k\in \{1,\dots, q\}$ with $ 1\leq k\leq n+1,$ the linear forms $ L_{i_1},\dots, L_{i_k}$ are linearly independent over $\mathbb{C}.$
\end{defn}

Let $\{H_j\}_{j=1}^q $ and $\{L_j\}_{j=1}^q $ be as above. If $ q\geq n+1,$ then $\{H_j\}_{j=1}^q $ is in general position if and only if
\[
   \det\big(a^{i_k}_0,\dots,a^{i_k}_n;\, 1\leq k\leq n+1\big)\neq 0
\]
for any $(n+1)$ distinct indices $ i_1,\dots, i_{n+1}\in\{1,\dots,q\}.$

\begin{defn}      \label{defn:pullback_divisor}
Let $ f:\mathbb{C}^m\to\mathbb{P}^n $ be a meromorphic map and let $ H $ be a hyperplane in $\mathbb{P}^n.$ By taking a reduced representation $(f_0,\dots,f_n)$ of $ f $ and taking a linear form $ a_0 X_0+\dots+a_n X_n $ that defines $ H,$ we define
 \[
    (f,H):=a_0f_0+\dots+a_nf_n,
 \]
which is a holomorphic function on $\mathbb{C}^m.$ When $ f(\mathbb{C}^m)\not\subseteq H,$ the pullback divisor $ f^*H $ is defined to be the zero divisor of the holomorphic function $(f,H),$ which is independent of the different choices of reduced representations of $ f $ and linear forms that define $ H.$
\end{defn}

\paragraph{2.2\, The Borel Lemma.}
\begin{thm}[Borel Lemma, see \cite{Borel1897} and Corollary 4.2 in \cite{Fujimoto75}]       \label{thm:Borel_Lemma}
Let $ h_1,\dots, h_q $ be nowhere zero holomorphic functions on $\mathbb{C}^m $ which satisfy the following equation:
 \[
    c_1h_1+\dots+c_qh_q \equiv 0,
 \]
where $ c_1,\dots, c_q $ are nonzero complex numbers. Then, for each $ h_i,$ there is some $ h_j $ with $ j\neq i $ such that $ h_j/h_i $ is constant.
\end{thm}

We now give some notations.
\begin{notation}
Denote by $\mathcal{H}^*=\mathcal{H}^*_m $ the multiplicative group of all nowhere zero holomorphic functions on $\mathbb{C}^m.$ We regard the set $\mathbb{C}^*$ of all nonzero complex numbers as a subgroup of $\mathcal{H}^*.$ For any $ h\in\mathcal{H}^*,$ we denote by $[h]$ the equivalence class in the quotient group $\mathcal{H}^*/\mathbb{C}^*$ that contains $ h.$
\end{notation}

We emphasize that the quotient group $\mathcal{H}^*/\mathbb{C}^*$ is a \emph{torsion-free} abelian group.

A corollary of the Borel Lemma is the following.
\begin{prop}[see Proposition 4.5 in \cite{Fujimoto75}]      \label{prop:multi._independe._implies_algebrai._independe.}
Let $\eta_1,\dots,\eta_t $ be nowhere zero holomorphic functions on $\mathbb{C}^m $ and assume they are multiplicatively independent, namely,
 \[
    \eta_1^{n_1}\cdot\eta_2^{n_2}\cdots\eta_t^{n_t}\not\in \mathbb{C}^*
 \]
for any $(n_1,\dots,n_t)\in\mathbb{Z}^t\setminus\{(0,\dots,0)\}.$ If a polynomial $ P(X_1,\dots, X_t)\in\mathbb{C}[X_1,\dots,X_t]$ satisfies
 \[
    P(\eta_1,\dots,\eta_t)\equiv 0,
 \]
then $ P(X_1,\dots, X_t)$ is the zero polynomial.
\end{prop}

\paragraph{2.3\, Combinatorial lemmas.}

Let $ G $ be a torsion-free abelian group. Let $ A=(\alpha_1,\dots,\alpha_q)$ be a $ q$-tuple of elements in $ G.$ We denote by $\langle\alpha_1,\dots, \alpha_q\rangle$ the subgroup of $ G $ generated by $\alpha_1,\dots, \alpha_q.$ Because $ G $ is torsion-free, $\langle\alpha_1,\dots, \alpha_q\rangle$ is free and is of finite rank. We denote by {\rm rank}$\{\alpha_1,\dots, \alpha_q\}$ the rank of the free abelian group $\langle\alpha_1,\dots, \alpha_q\rangle.$

The following definition is due to Fujimoto \cite{Fujimoto75}.
\begin{defn}      \label{defn:property_(Pr,s)}
Let $(G,\cdot)$ be a torsion-free abelian group. Let $ q\geq r>s\geq 1 $ be integers. A $ q$-tuple $ A=(\alpha_1,\dots,\alpha_q)$ of elements in $ G $ is said to \emph{have the property} $(P_{r,s})$ if arbitrarily chosen $ r $ elements $\alpha_{l(1)},\dots,\alpha_{l(r)}$ in $ A $ ($ 1\leq l(1)<\dots<l(r)\leq q $) satisfy the condition that, for any $ s $ distinct indices $ i_1,\dots, i_s \in\{1,\dots,r\},$ there exist distinct indices $ j_1,\dots, j_s \in\{1,\dots,r\}$ with $\{j_1,\dots,j_s\}\neq \{i_1,\dots,i_s\}$ such that
 \[
    \alpha_{l(i_1)}\cdot\alpha_{l(i_2)}\cdots \alpha_{l(i_s)}= \alpha_{l(j_1)}\cdot\alpha_{l(j_2)}\cdots \alpha_{l(j_s)}.
 \]
\end{defn}

Fujimoto \cite{Fujimoto75} proved the following result.
\begin{lem}[see Lemma 2.6 in \cite{Fujimoto75}]      \label{lem:1stCombiLem}
Let $(G,\cdot)$ be a torsion-free abelian group. Let $ A=(\alpha_1,\dots,\alpha_q)$ be a $ q$-tuple of elements in $ G $ that has the property $(P_{r,s}),$ where $ q\geq r>s\geq 1.$ Then there exist $(q-r+2)$ distinct indices $ i_1,\dots,i_{q-r+2}\in\{1,\dots,q\}$ such that
 \[
    \alpha_{i_1}=\alpha_{i_2}=\dots=\alpha_{i_{q-r+2}}.
 \]
\end{lem}

To state another lemma, we need the following notation.
\begin{notation}
For elements $\alpha_1,\alpha_2,\dots,\alpha_q $, $\tilde{\alpha}_1,\tilde{\alpha}_2,\dots,\tilde{\alpha}_q $ in an abelian group $(G,\cdot),$ by the notation
 \[
    \alpha_1:\alpha_2:\dots:\alpha_q=\tilde{\alpha}_1:\tilde{\alpha}_2:\dots:\tilde{\alpha}_q,
 \]
we mean that $\alpha_i=\beta\tilde{\alpha}_i $ ($ 1\leq i\leq q $) for some element $\beta\in G.$
\end{notation}

Now we recall the following lemma.
\begin{lem}[see Lemma 3.6 in \cite{Fujimoto76} and \cite{zk23_ANote}]      \label{lem:3rdCombiLem}
Let $ s $ and $ q $ be integers with $ 2\leq s<q\leq 2s.$ Let $(G,\cdot)$ be a torsion-free abelian group. Let $ A=(\alpha_1,\dots,\alpha_q)$ be a $ q$-tuple of elements in $ G $ that has the property $(P_{q,s}),$ and assume that at least one $\alpha_i $ equals the unit element $ 1 $ of $ G.$ Then
 \begin{enumerate}[\rm (i)]
   \item {\rm rank}$\{\alpha_1,\dots,\alpha_q\}=:t\leq s-1;$
   \item if $ t=s-1,$ then $ q=2s $ and there is a basis $\{\beta_1,\dots,\beta_{s-1}\}$ of $\langle\alpha_1,\dots,\alpha_q\rangle$ such that the $\alpha_i $ are represented, after a suitable change of indices, as one of the following two types:
     \begin{itemize}
       \item[\rm (A)] $ s $ is odd and
            \[
               \alpha_1:\alpha_2:\dots:\alpha_{2s}= 1:1:\beta_1:\beta_1:\beta_2:\beta_2:\dots:\beta_{s-1}:\beta_{s-1};
            \]
       \item[\rm (B)] $\alpha_1:\alpha_2:\dots:\alpha_{2s}= 1:1:\dots:1:\beta_1:\dots:\beta_{s-1}: (\beta_1\cdots\beta_{a_1})^{-1}:(\beta_{a_1+1}\cdots\beta_{a_2})^{-1}:\dots:(\beta_{a_{k-1}+1}\cdots\beta_{a_k})^{-1},$
           \par where $ 0\leq k\leq s-1 $, $ 1\leq a_1<a_2<\dots<a_k\leq s-1,$ and the unit element $ 1 $ appears $(s+1-k)$ times in the right hand side.
     \end{itemize}
 \end{enumerate}
\end{lem}

We also have the following easy observation.
\begin{observation}      \label{obsv:two_proportional_tuples_with_unit_generate_same_subgroup}
Let $(\alpha_1,\dots,\alpha_q)$ and $(\tilde{\alpha}_1,\dots,\tilde{\alpha}_q)$ be two $ q$-tuples of elements in a torsion-free abelian group $(G,\cdot).$ Assume that there exist indices $ i_0 $ and $ j_0 $ such that $\alpha_{i_0}=\tilde{\alpha}_{j_0}=1,$ and there is an element $\beta\in G $ such that $\alpha_i= \beta\tilde{\alpha}_i $ for any $ 1\leq i\leq q.$ Then
 \[
    \langle\alpha_1,\dots,\alpha_q \rangle= \langle\tilde{\alpha}_1,\dots,\tilde{\alpha}_q \rangle.
 \]
\end{observation}

\section{The Zariski closure of the image of $f\times g$ in $\mathbb{P}^n\times\mathbb{P}^n$}     \label{sec:Vfg}

Let $ f $ and $ g $ be two meromorphic maps of $\mathbb{C}^m $ into $\mathbb{P}^n $ and let $\{H_j\}_{j=1}^{2n+2}$ be a family of hyperplanes in $\mathbb{P}^n $ in general position. Assume that $ f(\mathbb{C}^m)\not\subseteq H_j $, $ g(\mathbb{C}^m)\not\subseteq H_j,$ and $ f^*(H_j)=g^*(H_j)$ for $ 1\leq j\leq 2n+2.$

Let $ f\times g $ be the holomorphic map of $\mathbb{C}^m\setminus\big(I(f)\cup I(g)\big)$ into $\mathbb{P}^n\times\mathbb{P}^n $
that is given by $(f\times g)(z)=(f(z),g(z))$ for $ z\in\mathbb{C}^m\setminus\big(I(f)\cup I(g)\big),$
where $ I(f)$ and $ I(g)$ are the indeterminacy loci of $ f $ and $ g,$ respectively.

\begin{defn}      \label{defn:Vfg}
We define $ V_{f\times g}$ to be the Zariski closure of the image of $ f\times g $ in $\mathbb{P}^n\times\mathbb{P}^n,$
namely, the intersection of all algebraic sets in $\mathbb{P}^n\times\mathbb{P}^n $ which contain the image of $ f\times g.$
\end{defn}

We have the following easy proposition.
\begin{prop}    \label{prop:Vfg_is_irreducible}
The algebraic set $ V_{f\times g}$ is irreducible.
\end{prop}

\begin{proof}
Assume that $ V_{f\times g}$ is reducible, namely, $ V_{f\times g}=V_1\cup V_2 $ for two proper algebraic subsets $ V_1 $ and $ V_2 $ of $ V_{f\times g}.$
Then $ A_1:=(f\times g)^{-1}(V_1)$ and $ A_2:=(f\times g)^{-1}(V_2)$ are both analytic subsets of $\mathbb{C}^m\setminus\big(I(f)\cup I(g)\big),$ and
\[
   \mathbb{C}^m\setminus\big(I(f)\cup I(g)\big)=A_1\cup A_2.
\]
By the definition of $ V_{f\times g},$ we see that $ A_1 $ and $ A_2 $ are both proper analytic subsets of $\mathbb{C}^m\setminus\big(I(f)\cup I(g)\big),$ which contradicts the fact that $\mathbb{C}^m\setminus\big(I(f)\cup I(g)\big)$ is irreducible. This proves Proposition \ref{prop:Vfg_is_irreducible}.
\end{proof}

Take reduced representations $(f_0,\dots,f_n)$ and $(g_0,\dots,g_n)$ of $ f $ and $ g,$ respectively. Take a linear form $ a^j_0 X_0+\dots+a^j_n X_n $ that defines $ H_j $ for each $ 1\leq j\leq 2n+2.$ Define, for each $ 1\leq i\leq 2n+2,$
\begin{equation}      \label{equ:definition_of_hi}
   h_i:=\frac{(f,H_i)}{(g,H_i)},
\end{equation}
where
\[
   (f,H_i)=a^i_0 f_0+\dots+a^i_n f_n \quad\mbox{and}\quad (g,H_i)=a^i_0 g_0+\dots+a^i_n g_n.
\]
By assumption, each $ h_i $ is a nowhere zero holomorphic function on $\mathbb{C}^m.$ By choosing a new reduced representation of $ f $ if necessary, we may assume that at least one $ h_i $ is constant.

Consider the $(2n+2)$-tuple $([h_1],\dots,[h_{2n+2}])$ of elements in $\mathcal{H}^*/\mathbb{C}^*.$
Define
\[
   t:={\rm rank}\{[h_1],\dots,[h_{2n+2}]\}.
\]
(We refer the reader to Section \ref{sec:Preliminaries_auxiliary results} for notations.)

\begin{remark}       \label{remark:t_is_independent_of_choices_of_representations}
The number $ t $ is independent of the different choices of reduced representations of $ f $ and $ g $ and linear forms that define $ H_j $'s, as long as at least one $ h_j $ is constant.
\end{remark}

Using Fujimoto's method in \cite{Fujimoto75}, one has the following.
\begin{prop}        \label{prop:hi_has_property_P(2n+2,n+1)}
The $(2n+2)$-tuple $([h_1],\dots,[h_{2n+2}])$ has the property $(P_{2n+2,n+1}).$
\end{prop}

From the above proposition and the conclusion (i) of Lemma \ref{lem:3rdCombiLem}, one gets the following.
\begin{prop}     \label{prop:0<=t<=n}
$ 0\leq t\leq n.$
\end{prop}

Now we give the main theorem of this section. Let the notations be as above.

\begin{thm}       \label{thm:[hi]doesnot(P2s,s)_implies_dimVfg<=n-s+t}
Suppose that the $(2n+2)$-tuple $([h_1],\dots,[h_{2n+2}])$ does not have the property $(P_{2s,s})$ for some positive integer $ s $ with $ 1\leq s\leq n+1.$ Then
 \[
    \max\{t, \dim V_{f\times g}\}\leq n-s+t.
 \]
\end{thm}

\begin{proof}
Clearly we have $ 1\leq s\leq n $ and $ t\geq 1.$

Take functions $\eta_1,\dots,\eta_t\in\mathcal{H}^*$ such that $\{[\eta_1],\dots,[\eta_t]\}$ is a basis of $\langle [h_1],\dots,[h_{2n+2}]\rangle.$
Then the functions $ h_i $ can be represented as follows:
\begin{equation}       \label{equ:represent_hi_by_eta1_dots_etat}
   h_i=c_i\eta_1^{l(i,1)}\cdots\eta_t^{l(i,t)}, \quad 1\leq i\leq 2n+2,
\end{equation}
where $ c_i\in\mathbb{C}^*$ and $ l(i,1),\dots,l(i,t)$ are integers.

By the assumption, there are $ 2s $ distinct indices $ i_1,\dots,i_{2s}$ such that the $ 2s$-tuple $([h_{i_1}],\dots,[h_{i_{2s}}])$ does not have the property $(P_{2s,s}).$ Without loss of generality, we may assume that the $ 2s$-tuple $([h_1],\dots,[h_{2s}])$ does not have the property $(P_{2s,s}).$

By making a transformation of homogeneous coordinates, we may assume that
\[
   (a^k_0,\dots,a^k_n)=(0,\dots,0,\underset{\underset{k-{\rm th}}{\uparrow}}{1},0,\dots,0), \quad 1\leq k\leq s,
\]
and
\[
   (a^{2s+k}_0,\dots,a^{2s+k}_n)=(0,\dots,0,\underset{\underset{(s+k)-{\rm th}}{\uparrow}}{1},0,\dots,0), \quad 1\leq k\leq n+1-s.
\]
After this transformation, we see that the functions $ f_0,\dots,f_n $, $ g_0,\dots,g_n $ are all not identically zero and that any minor of order $ s $ of the matrix $(a^i_j;\, 1\leq i\leq 2s, 0\leq j\leq s-1)$ does not vanish.

We claim that
\begin{equation}      \label{equ:det(aij,aijhi;1<=i<=2s)neq0}
\det\big(a^i_0,\dots,a^i_{s-1},a^i_0h_i,\dots,a^i_{s-1}h_i;\, 1\leq i\leq 2s\big)\not\equiv 0.
\end{equation}
For, if the left hand side of \eqref{equ:det(aij,aijhi;1<=i<=2s)neq0} is identically zero, then by Fujimoto's method in \cite{Fujimoto75} which uses the Laplace expansion formula and the Borel Lemma we conclude that $([h_1],\dots,[h_{2s}])$ has the property $(P_{2s,s}),$ which contradicts the assumption.

Since $ H_1,\dots,H_s,H_{2s+1},\dots,H_{n+s+1}$ are the coordinate hyperplanes, we see from the definition of $ h_i $ that
\begin{equation}      \label{equ:fi=hjgi_in_proof_of_dimVfg<=...}
   f_j=h_{j+1}g_j,\,\, 0\leq j\leq s-1, \quad\mbox{and}\quad f_j=h_{j+s+1}g_j,\,\, s\leq j\leq n.
\end{equation}

By the definition of $ h_i,$ we get the following identities:
\[
   a^i_0f_0+\dots+a^i_{s-1}f_{s-1}- a^i_0h_ig_0-\dots-a^i_{s-1}h_ig_{s-1}=b_i, \quad 1\leq i\leq 2s,
\]
where
\begin{align*}
   b_i &=-a^i_sf_s-\dots-a^i_nf_n+a^i_sh_ig_s+\dots+a^i_nh_ig_n
\\ (\mbox{by}\,\, \eqref{equ:fi=hjgi_in_proof_of_dimVfg<=...})\quad &= a^i_s(h_i-h_{2s+1})g_s+\dots+a^i_n(h_i-h_{n+s+1})g_n.
\end{align*}
By \eqref{equ:det(aij,aijhi;1<=i<=2s)neq0}, we can use the Cramer's rule to get the following expressions of $ g_i\, (0\leq i\leq s-1):$
\[
   g_i=\frac{\tilde{P}_{i,s}(h_1,\dots,h_{2n+2})g_s+\dots+\tilde{P}_{i,n}(h_1,\dots,h_{2n+2})g_n}{\tilde{P}(h_1,\dots,h_{2s})}, \quad 0\leq i\leq s-1,
\]
where $\tilde{P}_{i,s}(X_1,\dots,X_{2n+2}),\dots,\tilde{P}_{i,n}(X_1,\dots,X_{2n+2})$ and $\tilde{P}(X_1,\dots,X_{2s})$ are polynomials. Then by \eqref{equ:represent_hi_by_eta1_dots_etat}, we see that the functions $ g_i\, (0\leq i\leq s-1)$ can be expressed as follows:
\begin{equation}     \label{equ:express_gi_by_gs_dots_gn_and_eta1_dots_etat}
   g_i=\frac{P_{i,s}(1,\eta_1,\dots,\eta_t)g_s+\dots+P_{i,n}(1,\eta_1,\dots,\eta_t)g_n}{P(1,\eta_1,\dots,\eta_t)}, \quad 0\leq i\leq s-1,
\end{equation}
where $ P(z_0,z_1,\dots,z_t)$ is a homogeneous polynomial of degree $ d\geq 1,$ and each $ P_{i,j}(z_0,z_1,\dots,z_t)\,(0\leq i\leq s-1, s\leq j\leq n)$ is either zero or a homogeneous polynomial of degree $ d.$

By \eqref{equ:fi=hjgi_in_proof_of_dimVfg<=...} and \eqref{equ:represent_hi_by_eta1_dots_etat}, we see that there are monomials $ Q_i(z_0,\dots,z_t)$ and $ R_i(z_0,\dots,z_t)\,(0\leq i\leq n-1)$ with $\deg Q_i=\deg R_i $ such that
\begin{equation}    \label{equ:express_fi_by_gs_dots_gn_and_eta1_dots_etat}
   \frac{f_i}{f_n}=\frac{Q_i(1,\eta_1,\dots,\eta_t)}{R_i(1,\eta_1,\dots,\eta_t)}\cdot\frac{g_i}{g_n}, \quad 0\leq i\leq n-1.
\end{equation}

Let $\eta $ be the holomorphic map of $\mathbb{C}^m $ into $\mathbb{P}^t $ that is given by
\[
   \eta(z)=[1:\eta_1(z):\dots:\eta_t(z)], \quad z\in\mathbb{C}^m.
\]
And let $ f\times g\times \eta $ be the holomorphic map of $\mathbb{C}^m\setminus\big(I(f)\cup I(g)\big)$ into $\mathbb{P}^n\times\mathbb{P}^n\times\mathbb{P}^t $
that is given by
\[
   (f\times g\times \eta)(z)=(f(z),g(z),\eta(z)), \quad z\in\mathbb{C}^m\setminus\big(I(f)\cup I(g)\big).
\]
We denote by $ V_{f,g,\eta}$ the Zariski closure of the image of $ f\times g\times \eta $ in $\mathbb{P}^n\times\mathbb{P}^n\times\mathbb{P}^t.$ The same argument as in the proof of Proposition \ref{prop:Vfg_is_irreducible} shows that $ V_{f,g,\eta}$ is irreducible.

Put
\begin{align*}
   V:&=\big\{\big([x_0:\dots:x_n],[y_0:\dots:y_n],[z_0:\dots:z_t]\big)\in\mathbb{P}^n\times\mathbb{P}^n\times\mathbb{P}^t\, |\,
\\ & \hspace{2em} P(z_0,\dots,z_t)y_i-P_{i,s}(z_0,\dots,z_t)y_s-\dots-P_{i,n}(z_0,\dots,z_t)y_n=0,
\\ & \hspace{24em} 0\leq i\leq s-1,
\\ & \hspace{5em} x_iy_n R_i(z_0,\dots,z_t)- x_ny_i Q_i(z_0,\dots,z_t)=0, \quad 0\leq i\leq n-1\big\}.
\end{align*}
Then $ V $ is an algebraic set in $\mathbb{P}^n\times\mathbb{P}^n\times\mathbb{P}^t,$ and by \eqref{equ:express_gi_by_gs_dots_gn_and_eta1_dots_etat} and \eqref{equ:express_fi_by_gs_dots_gn_and_eta1_dots_etat} we see that
\[
   V_{f,g,\eta}\subseteq V.
\]
Let $ U $ be the subset of $ V $ which contains all points $\big([x_0:\dots:x_n],[y_0:\dots:y_n],[z_0:\dots:z_t]\big)$ in $ V $ such that
\[
   x_n\neq 0,\quad y_n\neq 0,\quad z_0\neq 0, \quad P(z_0,\dots,z_t)\neq 0,
\]
and
\[
   R_i(z_0,\dots,z_t)\neq 0,\quad 0\leq i\leq n-1.
\]
Then $ U $ is a Zariski-open subset of $ V.$ Since $ f_n $, $ g_n $, $ P(1,\eta_1,\dots,\eta_t)$, and $ R_i(1,\eta_1,\dots,\eta_t)\,(0\leq i\leq n-1)$ are all nonzero holomorphic functions on $\mathbb{C}^m,$ we see easily that
\[
   V_{f,g,\eta}\cap U\neq \emptyset.
\]

Let $\phi:U\to\mathbb{C}^{n-s+t}$ be the map
\[
   \big([x_0:\dots:x_n],[y_0:\dots:y_n],[z_0:\dots:z_t]\big)\mapsto \Big(\frac{y_s}{y_n},\dots,\frac{y_{n-1}}{y_n},\frac{z_1}{z_0},\dots,\frac{z_t}{z_0}\Big).
\]
By the definitions of $ V $ and $ U,$ we see easily that $\phi $ is injective. Since $ V_{f,g,\eta}$ is irreducible and $ V_{f,g,\eta}\cap U $ is a nonempty Zariski-open subset of $ V_{f,g,\eta},$ we conclude from the injectivity of $\phi|_{V_{f,g,\eta}\cap U}:V_{f,g,\eta}\cap U\to \mathbb{C}^{n-s+t}$ that
\[
   \dim V_{f,g,\eta}\leq n-s+t.
\]

Let $\pi_{1,2}:\mathbb{P}^n\times\mathbb{P}^n\times\mathbb{P}^t\to \mathbb{P}^n\times\mathbb{P}^n $ be the projection onto the first two components and let $\pi_3:\mathbb{P}^n\times\mathbb{P}^n\times\mathbb{P}^t\to \mathbb{P}^t $ be the projection onto the third component. Since $\pi_{1,2}(V_{f,g,\eta})$ contains the image of $ f\times g,$ we have
\[
   V_{f\times g}\subseteq\pi_{1,2}(V_{f,g,\eta}).
\]
Thus
\[
   \dim V_{f\times g}\leq \dim\pi_{1,2}(V_{f,g,\eta})\leq \dim V_{f,g,\eta}.
\]
Because $\eta_1,\dots,\eta_t $ are multiplicatively independent functions in $\mathcal{H}^*,$ we see by Proposition \ref{prop:multi._independe._implies_algebrai._independe.} that the holomorphic map $\eta:\mathbb{C}^m\to\mathbb{P}^t $ is algebraically non-degenerate. Then since $\pi_3(V_{f,g,\eta})\supseteq\eta(\mathbb{C}^m),$ we get $\pi_3(V_{f,g,\eta})=\mathbb{P}^t,$ and thus
\[
   t\leq \dim V_{f,g,\eta}.
\]
It follows from the above inequalities that
\[
   \max\{t,\dim V_{f\times g}\}\leq \dim V_{f,g,\eta}\leq n-s+t,
\]
which proves Theorem \ref{thm:[hi]doesnot(P2s,s)_implies_dimVfg<=n-s+t}.
\end{proof}

\section{Two propositions}          \label{sec:Two_propositions}

Let $ f $ and $ g $ be two meromorphic maps of $\mathbb{C}^m $ into $\mathbb{P}^n $ and let $\{H_j\}_{j=1}^{2n+2}$ be a family of hyperplanes in $\mathbb{P}^n $ in general position. Assume that $ f(\mathbb{C}^m)\not\subseteq H_j $, $ g(\mathbb{C}^m)\not\subseteq H_j,$ and $ f^*(H_j)=g^*(H_j)$ for $ 1\leq j\leq 2n+2.$

Take reduced representations $(f_0,\dots,f_n)$ and $(g_0,\dots,g_n)$ of $ f $ and $ g,$ respectively. Take a linear form $ a^j_0 X_0+\dots+a^j_n X_n $ that defines $ H_j $ for each $ 1\leq j\leq 2n+2.$ Define $ h_i $ as \eqref{equ:definition_of_hi} for each $ 1\leq i\leq 2n+2.$ Then each $ h_i $ is a nowhere zero holomorphic function on $\mathbb{C}^m.$ By Proposition \ref{prop:hi_has_property_P(2n+2,n+1)}, the $(2n+2)$-tuple $([h_1],\dots,[h_{2n+2}])$ has the property $(P_{2n+2,n+1}).$

\begin{prop}       \label{prop:[h1]=...=[h(n+2)]_and..._imply_f=g}
If $ g $ is linearly non-degenerate and there are $(n+2)$ distinct indices $ i_1,\dots,i_{n+2}$ such that
 \[
    [h_{i_1}]=[h_{i_2}]=\dots=[h_{i_{n+2}}],
 \]
then $ f=g.$
\end{prop}

\begin{proof}
The proof is essentially contained in \cite[pp. 11--12]{Fujimoto75}. We refer the reader to \cite{Fujimoto75} for the proof.
\end{proof}

\begin{prop}       \label{prop:hi=h(i+n+1)_implies_f=g}
Assume $ g(\mathbb{C}^m)$ is not contained in any hypersurface of degree $\leq (n+1)$ in $\mathbb{P}^n.$ If there is a bijection $\sigma: \{1,\dots,2n+2\}\to\{1,\dots,2n+2\}$ such that
 \[
    [h_{\sigma(i)}]=[h_{\sigma(i+n+1)}],\quad 1\leq i\leq n+1,
 \]
then $ f=g.$
\end{prop}

\begin{proof} (The argument is essentially same as in \cite[pp. 139--140]{Fujimoto76}.)

Without loss of generality, we may assume that
\[
   [h_i]=[h_{i+n+1}],\quad 1\leq i\leq n+1,
\]
namely,
\[
   h_i=c_{i-1}h_{i+n+1},\quad 1\leq i\leq n+1,
\]
for some nonzero constants $ c_0,\dots,c_n.$

By making a transformation of homogeneous coordinates, we may assume that
\[
   (a^{i+n+1}_0,\dots,a^{i+n+1}_n)=(0,\dots,0,\underset{\underset{i-{\rm th}}{\uparrow}}{1},0,\dots,0), \quad 1\leq i\leq n+1.
\]
After this transformation, we see from the definition of $ h_j $ that
\[
   \frac{a^{i+1}_0f_0+\dots+a^{i+1}_nf_n}{a^{i+1}_0g_0+\dots+a^{i+1}_ng_n}=c_i\frac{f_i}{g_i}, \quad 0\leq i\leq n.
\]
From these, we get the following identities:
\begin{equation}        \label{equ:bi0f0+dots+binfn=0_in_proof_hi=h(i+n+1)}
   b^i_0f_0+\dots+b^i_nf_n=0, \quad 0\leq i\leq n,
\end{equation}
where
\[
   b^i_i=a^{i+1}_ig_i-c_i(a^{i+1}_0g_0+\dots+a^{i+1}_ng_n)
\]
and
\[
   b^i_j=a^{i+1}_jg_i, \quad j\in\{0,\dots,n\}\setminus\{i\}.
\]
Eliminating $ f_0,\dots, f_n $ from the above identities, we get
\begin{equation}        \label{equ:det(bij(gi))=0_in_proof_hi=h(i+n+1)}
   \det(b^i_0,\dots, b^i_n;\, 0\leq i\leq n)\equiv 0.
\end{equation}
Noting that each $ b^i_j $ can be expressed as a homogeneous polynomial of degree $ 1 $ in $ g_0,\dots,g_n,$ by the assumption, we conclude that the left hand side of \eqref{equ:det(bij(gi))=0_in_proof_hi=h(i+n+1)} vanishes as a determinant over the polynomial ring $\mathbb{C}[g_0,\dots,g_n].$ Then, for an index $ k\in\{0,\dots,n\},$ by putting $ g_k=1 $ and $ g_j=0\, (j\neq k),$ we get
\[
   (1-c_k)a^{k+1}_k\cdot(-c_0a^1_k)\cdot\cdots\cdot(-c_{k-1}a^k_k)\cdot(-c_{k+1}a^{k+2}_k)\cdot\cdots\cdot(-c_na^{n+1}_k)=0,
\]
which implies that $ c_k=1.$ Therefore $ c_0=c_1=\dots=c_n=1.$

Consider the following system of linear equations over the field of meromorphic functions on $\mathbb{C}^m $:
\[
   (**) \qquad b^i_0x_0+\dots+b^i_{n-1}x_{n-1}=-b^i_nf_n, \quad 0\leq i\leq n-1.
\]
Its coefficient determinant $\det(b^i_0,\dots, b^i_{n-1};\, 0\leq i\leq n-1)$ can be expressed as a nonzero homogeneous polynomial of degree $ n $ in $ g_0,\dots,g_n.$ Thus by the assumption,
\[
   \det(b^i_0,\dots, b^i_{n-1};\, 0\leq i\leq n-1)\not\equiv 0,
\]
which shows that $(**)$ has a unique solution. By \eqref{equ:bi0f0+dots+binfn=0_in_proof_hi=h(i+n+1)}, we see that
\[
   (x_0,\dots,x_{n-1})=(f_0,\dots,f_{n-1})
\]
is a solution of $(**).$ On the other hand, by $ c_0=\dots=c_{n-1}=1,$ we get
\[
   b^i_i=-\sum_{j=0,j\neq i}^n a^{i+1}_jg_j, \quad 0\leq i\leq n-1,
\]
and then one easily verifies that
\[
   (x_0,\dots,x_{n-1})=\Big(\frac{f_n}{g_n}g_0,\dots,\frac{f_n}{g_n}g_{n-1}\Big)
\]
is also a solution of $(**).$ Therefore
\[
   f_i=\frac{f_n}{g_n}g_i, \quad 0\leq i\leq n-1,
\]
which implies $ f=g.$ This finishes the proof of Proposition \ref{prop:hi=h(i+n+1)_implies_f=g}.
\end{proof}

\section{The uniqueness theorem with generic $(2n+2)$ hyperplanes}           \label{sec:uniqueness_theorem_generic_(2n+2)}

We shall denote by $(\mathbb{P}^n)^*$ the set of all hyperplanes in $\mathbb{P}^n $ and denote by $((\mathbb{P}^n)^*)^{2n+2}$ the Cartesian product $\underset{(2n+2)\,\mbox{copies}}{\underbrace{(\mathbb{P}^n)^*\times\cdots\times(\mathbb{P}^n)^*}}.$

\begin{thm}           \label{thm:generic_(2n+2)_accurate}
There is a proper algebraic subset $ V $ of $((\mathbb{P}^n)^*)^{2n+2}$ with the following property:

Let $ H_1,\dots,H_{2n+2}$ be $(2n+2)$ hyperplanes in $\mathbb{P}^n $ in general position such that $(H_1,\dots,H_{2n+2})\not\in V.$ Let $ f $ and $ g $ be meromorphic maps of $\mathbb{C}^m $ into $\mathbb{P}^n $ such that $ f(\mathbb{C}^m)\not\subseteq H_j $, $ g(\mathbb{C}^m)\not\subseteq H_j $ and $ f^*(H_j)=g^*(H_j)$ $(1\leq j\leq 2n+2),$ and assume $ g $ is algebraically non-degenerate. Then $ f=g.$
\end{thm}

\begin{prop}    \label{prop:dimVfg>=n_when_g_is_algebrai_nondege}
Let $ f $ and $ g $ be meromorphic maps of $\mathbb{C}^m $ into $\mathbb{P}^n,$ and let $ V_{f\times g}$ be as in Definition {\rm\ref{defn:Vfg}}. If $ g $ is algebraically non-degenerate, then $\dim V_{f\times g}\geq n.$
\end{prop}

\begin{proof}
Let $\pi_2:\mathbb{P}^n\times\mathbb{P}^n\to\mathbb{P}^n $ be the projection onto the second component. Since $\pi_2(V_{f\times g})$ contains the image of $ g $ and $ g $ is algebraically non-degenerate, we see that $\pi_2(V_{f\times g})=\mathbb{P}^n.$ Therefore
\[
   n=\dim\pi_2(V_{f\times g})\leq \dim V_{f\times g},
\]
which proves Proposition \ref{prop:dimVfg>=n_when_g_is_algebrai_nondege}.
\end{proof}

\begin{proof}[Proof of Theorem {\rm \ref{thm:generic_(2n+2)_accurate}}]
Let $(f_0,\dots,f_n)$ and $(g_0,\dots,g_n)$ be reduced representations of $ f $ and $ g,$ respectively. Take a linear form $ a^j_0 X_0+\dots+a^j_n X_n $ that defines $ H_j $ for each $ 1\leq j\leq 2n+2.$ Define $ h_i $ as \eqref{equ:definition_of_hi} for each $ 1\leq i\leq 2n+2.$ By choosing a new reduced representation of $ f $ if necessary, we may assume at least one $ h_i $ is constant.
Consider the $(2n+2)$-tuple $([h_1],\dots,[h_{2n+2}])$ of elements in $\mathcal{H}^*/\mathbb{C}^*.$ Define $ t:={\rm rank}\{[h_1],\dots,[h_{2n+2}]\}.$ As in Section \ref{sec:Vfg}, we have $ 0\leq t\leq n.$

We assume that $ f\neq g.$ Then $ t\geq 1.$ For, if $ t=0 $ which means $[h_1]=[h_2]=\dots=[h_{2n+2}]=1,$ then Proposition \ref{prop:[h1]=...=[h(n+2)]_and..._imply_f=g} tells us that $ f=g.$

The proof contains several steps.

\vspace{1em}
\emph{Step 1.} Let $ V_{f\times g}$ be as in Definition \ref{defn:Vfg}. Proposition \ref{prop:dimVfg>=n_when_g_is_algebrai_nondege} tells us that $\dim V_{f\times g}\geq n $ because $ g $ is algebraically non-degenerate. It then follows from Theorem \ref{thm:[hi]doesnot(P2s,s)_implies_dimVfg<=n-s+t} that the $(2n+2)$-tuple $([h_1],\dots,[h_{2n+2}])$ has the property $(P_{2t+2,t+1}).$ For, if this is not true, we conclude by putting $ s=t+1 $ in Theorem \ref{thm:[hi]doesnot(P2s,s)_implies_dimVfg<=n-s+t} that
\[
   \dim V_{f\times g}\leq n-1,
\]
which contradicts the previous conclusion.

\vspace{1em}
\emph{Step 2.} By Lemma \ref{lem:1stCombiLem}, we may assume, after suitably choosing a new reduced representation of $ f $ and changing indices, that
\[
   [h_{2t+1}]=[h_{2t+2}]=\dots=[h_{2n+2}]=1.
\]
Since
\[
   {\rm rank}\{[h_1],\dots,[h_{2t+2}]\}=t,
\]
we can apply Lemma \ref{lem:3rdCombiLem} to the $(2t+2)$-tuple $([h_1],\dots,[h_{2t+2}])$ and conclude that there are multiplicatively independent elements $\beta_1,\dots,\beta_t\in \mathcal{H}^*/\mathbb{C}^*$ such that the $ h_i $ are represented, after a suitable change of indices, as one of the following two types:
\begin{enumerate}[\rm (a)]
  \item $([h_1],\dots,[h_{2n+2}])=(1,1,\dots,1,1,\beta_1,\beta_1,\dots,\beta_t,\beta_t);$
  \item $([h_1],\dots,[h_{2n+2}])$
         \par \hspace{2.5em} $=\big(1,1,\dots,1,\beta_1,\dots,\beta_t,(\beta_1\cdots\beta_{a_1})^{-1},\dots,(\beta_{a_{k-1}+1}\cdots\beta_{a_k})^{-1}\big),$
         \par where $ 0\leq k\leq t $ and $ 1\leq a_1<a_2<\dots<a_k\leq t.$
\end{enumerate}

\vspace{1em}
\emph{Step 3.} If the $ h_i $ are of the type (a), then Proposition \ref{prop:hi=h(i+n+1)_implies_f=g} shows that $ f=g,$ which contradicts our assumption. So the $ h_i $ are of the type (b).

We note that $ k\geq 1 $ because $ t\leq n.$ In fact, if $ k=0,$ then there are $ 2n+2-t(\geq n+2)$ distinct indices $ i_1,\dots,i_{2n+2-t}\in\{1,\dots,2n+2\}$ such that $[h_{i_1}]=[h_{i_2}]=\dots=[h_{i_{2n+2-t}}]=1,$ and thus Proposition \ref{prop:[h1]=...=[h(n+2)]_and..._imply_f=g} tells us that $ f=g,$ which contradicts our assumption.

Now, we may assume, after suitably changing indices and choosing a new reduced representation of $ f,$ that
\begin{align}       \label{equ:representations_of_hi_with_undetermined_constants}
   (h_1,\dots,h_{2n+2})\qquad &          \notag
\\   =\big(c_1(\eta_1\cdots\eta_{a_1}&)^{-1},c_2(\eta_{a_1+1}\cdots\eta_{a_2})^{-1},\dots,c_k(\eta_{a_{k-1}+1}\cdots\eta_{a_k})^{-1},      \notag
\\ &\hspace{8.5em} c_{k+1},\dots,c_{2n+1-t},1,\eta_1,\dots,\eta_t\big),
\end{align}
where $ 1\leq k\leq t $, $ 1\leq a_1<a_2<\dots<a_k\leq t $, $ c_1,\dots,c_{2n+1-t}$ are nonzero constants, and $\eta_1,\dots,\eta_t $ are multiplicatively independent functions in $\mathcal{H}^*.$

\vspace{1em}
\emph{Step 4.} Consider the $(2t+2)$ functions $ h_1,\dots,h_{t+1},h_{2n+2-t},\dots,h_{2n+2}.$
As in the proof of Theorem \ref{thm:[hi]doesnot(P2s,s)_implies_dimVfg<=n-s+t}, we make a transformation of homogeneous coordinates such that
\[
   (a^i_0,\dots,a^i_n)=(0,\dots,0,\underset{\underset{(i-2n-1+t)-{\rm th}}{\uparrow}}{1},0,\dots,0), \quad 2n+2-t\leq i\leq 2n+2,
\]
and
\[
   (a^i_0,\dots,a^i_n)=(0,\dots,0,\underset{\underset{i-{\rm th}}{\uparrow}}{1},0,\dots,0), \quad t+2\leq i\leq n+1.
\]
We note that, after this transformation, any minor of the matrix $(a^i_j;\, 1\leq i\leq t+1,0\leq j\leq n)$ does not vanish.

Since $\dim V_{f\times g}\geq n,$ using the essentially same argument as in the proof of Theorem \ref{thm:[hi]doesnot(P2s,s)_implies_dimVfg<=n-s+t}, we get
\begin{equation}       \label{equ:det()=0_as_an_implication_of_dimVfg>=n}
\det\big(a^i_0,\dots,a^i_t,a^i_0h_i,\dots,a^i_th_i;\, i=1,\dots,t+1,2n+2-t,\dots,2n+2\big)\equiv 0.
\end{equation}
The matrix $\big(a^i_0,\dots,a^i_t;\, i=2n+2-t,\dots,2n+2\big)$ is the identity matrix of order $(t+1),$ and by \eqref{equ:representations_of_hi_with_undetermined_constants} one sees that
\[
   \big(a^i_0h_i,\dots,a^i_th_i;\, i=2n+2-t,\dots,2n+2\big)={\rm diag}(1,\eta_1,\dots,\eta_t).
\]
Then, by \eqref{equ:representations_of_hi_with_undetermined_constants}, we conclude from \eqref{equ:det()=0_as_an_implication_of_dimVfg>=n} the following identity.
\begin{equation}        \label{equ:det(eta1_dots_etat)=0_as_an_implication_of_dimVfg>=n}
  \left|
    \begin{array}{cccc}
       a^1_0(c_1-\tilde{\eta}_1) & a^1_1(c_1-\tilde{\eta}_1\cdot\eta_1) & \cdots & a^1_t(c_1-\tilde{\eta}_1\cdot\eta_t)
    \\ a^2_0(c_2-\tilde{\eta}_2) & a^2_1(c_2-\tilde{\eta}_2\cdot\eta_1) & \cdots & a^2_t(c_2-\tilde{\eta}_2\cdot\eta_t)
    \\   \vdots & \vdots & \ddots & \vdots
    \\ a^k_0(c_k-\tilde{\eta}_k) & a^k_1(c_k-\tilde{\eta}_k\cdot\eta_1) & \cdots & a^k_t(c_k-\tilde{\eta}_k\cdot\eta_t)
    \\ a^{k+1}_0(c_{k+1}-1) & a^{k+1}_1(c_{k+1}-\eta_1) & \cdots & a^{k+1}_t(c_{k+1}-\eta_t)
    \\   \vdots & \vdots & \ddots & \vdots
    \\ a^{t+1}_0(c_{t+1}-1) & a^{t+1}_1(c_{t+1}-\eta_1) & \cdots & a^{t+1}_t(c_{t+1}-\eta_t)
    \end{array}
  \right|=0,
\end{equation}
where $\tilde{\eta}_1=\eta_1\dots\eta_{a_1}$, $\tilde{\eta}_2=\eta_{a_1+1}\dots\eta_{a_2}$, $\dots$, $\tilde{\eta}_k=\eta_{a_{k-1}+1}\dots\eta_{a_k}.$

Since $\eta_1,\dots,\eta_t $ are multiplicatively independent functions in $\mathcal{H}^*,$ using the Borel Lemma (see Proposition \ref{prop:multi._independe._implies_algebrai._independe.}), we conclude that the left hand side of \eqref{equ:det(eta1_dots_etat)=0_as_an_implication_of_dimVfg>=n} vanishes as a determinant over the polynomial ring $\mathbb{C}[\eta_1,\dots,\eta_t].$

\vspace{1em}
\emph{Step 5.} We shall use the following lemma due to Fujimoto \cite{Fujimoto78}.
\begin{lem}[see Lemma 4.1 in \cite{Fujimoto78}]       \label{lem:algebraic_lemma}
Let $ 1\leq k\leq t $ be positive integers and let $ 1\leq a_1<a_2<\dots<a_k\leq t $ be an increasing sequence of positive integers.
Let $(a^i_j;\, 1\leq i\leq t+1, 0\leq j\leq t)$ be a square matrix of order $(t+1)$ whose any minor does not vanish.
Let $ c_1,\dots,c_{t+1}$ be nonzero complex numbers. Assume the following equation holds:
 \begin{equation}        \label{equ:identity_of_polynomials_in_algebraic_lemma}
  \left|
    \begin{array}{cccc}
       a^1_0(c_1-\tilde{\eta}_1) & a^1_1(c_1-\tilde{\eta}_1\cdot\eta_1) & \cdots & a^1_t(c_1-\tilde{\eta}_1\cdot\eta_t)
    \\ a^2_0(c_2-\tilde{\eta}_2) & a^2_1(c_2-\tilde{\eta}_2\cdot\eta_1) & \cdots & a^2_t(c_2-\tilde{\eta}_2\cdot\eta_t)
    \\   \vdots & \vdots & \ddots & \vdots
    \\ a^k_0(c_k-\tilde{\eta}_k) & a^k_1(c_k-\tilde{\eta}_k\cdot\eta_1) & \cdots & a^k_t(c_k-\tilde{\eta}_k\cdot\eta_t)
    \\ a^{k+1}_0(c_{k+1}-1) & a^{k+1}_1(c_{k+1}-\eta_1) & \cdots & a^{k+1}_t(c_{k+1}-\eta_t)
    \\   \vdots & \vdots & \ddots & \vdots
    \\ a^{t+1}_0(c_{t+1}-1) & a^{t+1}_1(c_{t+1}-\eta_1) & \cdots & a^{t+1}_t(c_{t+1}-\eta_t)
    \end{array}
  \right|=0,
 \end{equation}
where the left hand side of the above equation is a determinant over the polynomial ring $\mathbb{C}[\eta_1,\dots,\eta_t],$
and $\tilde{\eta}_1=\eta_1\dots\eta_{a_1}$, $\tilde{\eta}_2=\eta_{a_1+1}\dots\eta_{a_2}$, $\dots$, $\tilde{\eta}_k=\eta_{a_{k-1}+1}\dots\eta_{a_k}.$
Then
 \begin{enumerate}[\rm (i)]
   \item $ a_k=t;$
   \item one of the following two cases occurs:
       \begin{itemize}
         \item[$(\alpha)$] $ k=1 $, $ c_1=1,$ and $\{c_2,\dots,c_{t+1}\}=\{\xi,\xi^2,\dots,\xi^t\},$ where $\xi=\exp(\frac{2\pi\sqrt{-1}}{t+1});$
         \item[$(\beta)$] $ k=t $, $ c_1=\dots=c_t=1,$ and $ c_{t+1}=(-1)^t.$
       \end{itemize}
 \end{enumerate}
\end{lem}

Applying Lemma \ref{lem:algebraic_lemma} to the equation \eqref{equ:det(eta1_dots_etat)=0_as_an_implication_of_dimVfg>=n}, we know that $ a_k=t $ and one of the following two cases occurs:
\begin{itemize}
  \item[$(\alpha)'$] $ k=1 $, $ c_1=1,$ and $\{c_2,\dots,c_{t+1}\}=\{\xi,\dots,\xi^t\},$ where $\xi=\exp(\frac{2\pi\sqrt{-1}}{t+1});$
  \item[$(\beta)'$] $ k=t $, $ c_1=\dots=c_t=1,$ and $ c_{t+1}=(-1)^t.$
\end{itemize}

\vspace{1em}
\emph{Step 6.} We shall show that $ t=n.$ To show $ t=n,$ we consider the following two cases.

\noindent$\bullet$ $ k=1.$

If $ t<n,$ then there are $ 2n+1-t(\geq n+2)$ distinct indices $ i_1,\dots,i_{2n+1-t}\in\{1,\dots,2n+2\}$ such that $[h_{i_1}]=[h_{i_2}]=\dots=[h_{i_{2n+1-t}}]=1,$ and this implies $ f=g,$ which contradicts our assumption. Hence $ t=n.$

\noindent$\bullet$ $ k\geq 2.$

Assume $ t<n.$ Consider the following $(2t+2)$ functions:
\[
   h_1,\dots,h_{k-1},h_{k+1},\dots,h_{t+2},h_{2n+2-t},\dots,h_{2n+2}.
\]
As in Step 4, by making a suitable transformation of homogeneous coordinates and using the argument in the proof of Theorem \ref{thm:[hi]doesnot(P2s,s)_implies_dimVfg<=n-s+t}, we conclude an equation which is similar to \eqref{equ:det(eta1_dots_etat)=0_as_an_implication_of_dimVfg>=n}. Then with the help of the Borel Lemma, we can apply Lemma \ref{lem:algebraic_lemma} to obtain $ a_{k-1}=t,$ which is a contradiction. Thus $ t=n.$

In any case, we get $ t=n.$

\vspace{1em}
\emph{Step 7.} Using the conclusions in Step 5 and Step 6, we see that the equation \eqref{equ:det(eta1_dots_etat)=0_as_an_implication_of_dimVfg>=n} becomes one of the following two forms:
\begin{equation}        \label{equ:det(eta1_dots_etan)=0_when_k=1}
  \left|
    \begin{array}{cccc}
       a^1_0(1-\eta_1\cdots \eta_n) & a^1_1(1-\eta_1\cdots\eta_n\cdot\eta_1) & \cdots & a^1_n(1-\eta_1\cdots\eta_n\cdot\eta_n)
    \\ a^2_0(\xi-1) & a^2_1(\xi-\eta_1) & \cdots & a^2_n(\xi-\eta_n)
    \\ a^3_0(\xi^2-1) & a^3_1(\xi^2-\eta_1) & \cdots & a^3_n(\xi^2-\eta_n)
    \\   \vdots & \vdots & \ddots & \vdots
    \\ a^{n+1}_0(\xi^n-1) & a^{n+1}_1(\xi^n-\eta_1) & \cdots & a^{n+1}_n(\xi^n-\eta_n)
    \end{array}
  \right|=0,
\end{equation}
where $\xi=\exp(\frac{2\pi\sqrt{-1}}{n+1})$ and the superscripts of the coefficients $ a^i_j $ are suitably exchanged, or
\begin{equation}        \label{equ:det(eta1_dots_etan)=0_when_k=t=n}
  \left|
    \begin{array}{cccc}
       a^1_0(1-\eta_1) & a^1_1(1-\eta_1\cdot\eta_1) & \cdots & a^1_n(1-\eta_1\cdot\eta_n)
    \\ a^2_0(1-\eta_2) & a^2_1(1-\eta_2\cdot\eta_1) & \cdots & a^2_n(1-\eta_2\cdot\eta_n)
    \\   \vdots & \vdots & \ddots & \vdots
    \\ a^n_0(1-\eta_n) & a^n_1(1-\eta_n\cdot\eta_1) & \cdots & a^n_n(1-\eta_n\cdot\eta_n)
    \\ a^{n+1}_0((-1)^n-1) & a^{n+1}_1((-1)^n-\eta_1) & \cdots & a^{n+1}_n((-1)^n-\eta_n)
    \end{array}
  \right|=0.
\end{equation}
We note that the left hand sides of the above two equations are both determinants of order $(n+1)$ over the polynomial ring $\mathbb{C}[\eta_1,\dots,\eta_n].$

By putting $\eta_1=\dots=\eta_n=2,$ we see that the equation \eqref{equ:det(eta1_dots_etan)=0_when_k=1} can be written as
\[
   P_1(a^1_0,\dots,a^1_n,a^2_0,\dots,a^2_n,\dots,a^{n+1}_0,\dots,a^{n+1}_n)=0
\]
and the equation \eqref{equ:det(eta1_dots_etan)=0_when_k=t=n} can be written as
\[
   P_2(a^1_0,\dots,a^1_n,a^2_0,\dots,a^2_n,\dots,a^{n+1}_0,\dots,a^{n+1}_n)=0,
\]
where both $ P_1 $ and $ P_2 $ are nonzero homogeneous polynomials of degree $(n+1)$ in $(n+1)^2 $ variables and are homogeneous polynomials of degree 1 in each block of variables $ a^i_0,\dots,a^i_n $ ($ i=1,\dots,n+1 $).

We note that the coefficients $ a^i_j $ in \eqref{equ:det(eta1_dots_etan)=0_when_k=1} and \eqref{equ:det(eta1_dots_etan)=0_when_k=t=n} are not the original coefficients of the linear forms that define the hyperplanes $ H_i,$ since we have made a transformation of homogeneous coordinates in Step 4. Using the original data of the coefficients of the linear forms that define $ H_i,$ we in fact have the following equations:
\[
   P_1\big((a^{\sigma(1)}_0,\dots,a^{\sigma(1)}_n)A_{\sigma}^{-1},\dots,(a^{\sigma(n+1)}_0,\dots,a^{\sigma(n+1)}_n)A_{\sigma}^{-1}\big)=0
\]
or
\[
   P_2\big((a^{\sigma(1)}_0,\dots,a^{\sigma(1)}_n)A_{\sigma}^{-1},\dots,(a^{\sigma(n+1)}_0,\dots,a^{\sigma(n+1)}_n)A_{\sigma}^{-1}\big)=0,
\]
where $\sigma:\{1,\dots,2n+2\}\to\{1,\dots,2n+2\}$ is a permutation and $ A_{\sigma}^{-1}$ is the inverse matrix of
\[
   A_{\sigma}=\big(a^{\sigma(i)}_0,\dots,a^{\sigma(i)}_n;\, n+2\leq i\leq 2n+2\big).
\]
Since $ A_{\sigma}^{-1}=A_{\sigma}^*/\det(A_{\sigma})$ where $ A_{\sigma}^*$ is the adjoint matrix of $ A_{\sigma},$ we conclude the following equations:
\begin{equation}      \label{equ:P1(aijA*,...)=0}
   P_1\big((a^{\sigma(1)}_0,\dots,a^{\sigma(1)}_n)A_{\sigma}^*,\dots,(a^{\sigma(n+1)}_0,\dots,a^{\sigma(n+1)}_n)A_{\sigma}^*\big)=0
\end{equation}
or
\begin{equation}      \label{equ:P2(aijA*,...)=0}
   P_2\big((a^{\sigma(1)}_0,\dots,a^{\sigma(1)}_n)A_{\sigma}^*,\dots,(a^{\sigma(n+1)}_0,\dots,a^{\sigma(n+1)}_n)A_{\sigma}^*\big)=0.
\end{equation}

One sees that the equation \eqref{equ:P1(aijA*,...)=0} defines a proper algebraic subset $ V_{1,\sigma}$ of $((\mathbb{P}^n)^*)^{2n+2}$ and the equation \eqref{equ:P2(aijA*,...)=0} defines a proper algebraic subset $ V_{2,\sigma}$ of $((\mathbb{P}^n)^*)^{2n+2}.$ Let $ V $ be the finite union
\[
   V:=\bigcup_{\mbox{permutation}\,\, \sigma} (V_{1,\sigma}\cup V_{2,\sigma}).
\]
Then $ V $ is also a proper algebraic subset of $((\mathbb{P}^n)^*)^{2n+2},$ and $(H_1,\dots,H_{2n+2})\in V.$

So if we take the $ V $ in the condition of Theorem \ref{thm:generic_(2n+2)_accurate} to be the above $ V $ and assume $(H_1,\dots,H_{2n+2})\not\in V,$ then we conclude $ f=g.$ This finishes the proof of Theorem \ref{thm:generic_(2n+2)_accurate}.
\end{proof}

\bibliographystyle{plain}
\bibliography{zkRef}

\begin{thebibliography}{1}

\bibitem{Borel1897}
Emile Borel.
\newblock Sur les z\'eros des fonctions enti\`eres.
\newblock {\em Acta Math.}, 20:357--396, 1897.

\bibitem{Fujimoto75}
Hirotaka Fujimoto.
\newblock The uniqueness problem of meromorphic maps into the complex
  projective space.
\newblock {\em Nagoya Math. J.}, 58:1--23, 1975.

\bibitem{Fujimoto76}
Hirotaka Fujimoto.
\newblock A uniqueness theorem of algebraically non-degenerate meromorphic maps
  into {$P^N(\mathbf{C})$}.
\newblock {\em Nagoya Math. J.}, 64:117--147, 1976.

\bibitem{Fujimoto78}
Hirotaka Fujimoto.
\newblock Remarks to the uniqueness problem of meromorphic maps into
  {$P^N(\mathbf{C})$, I}.
\newblock {\em Nagoya Math. J.}, 71:13--24, 1978.

\bibitem{Nevanlinna1926}
Rolf Nevanlinna.
\newblock {Einige Eindeutigkeitss\"atze in der Theorie der meromorphen
  Funktionen}.
\newblock {\em Acta Math.}, 48(3-4):367--391, 1926.

\bibitem{Polya1921}
George P\'olya.
\newblock {Bestimmung einer ganzen Funktion endlichen Geschlechts durch
  viererlei Stellen}.
\newblock {\em Mat. Tidsskrift B}, pages 16--21, 1921.
\newblock https://zbmath.org/?q=an\%3A48.0354.03.

\bibitem{zk23_ANote}
Kai Zhou.
\newblock A note on {Fujimoto's} uniqueness theorem with $(2n+3)$ hyperplanes.
\newblock 2023.
\newblock {\tt arXiv:2307.16595 [math.CV]}.

\end{thebibliography}

\end{document}